\documentclass{amsart}
\usepackage[left=2.5cm,    
right=2.5cm,   
top=2cm,       
bottom=2cm]{geometry} 
\usepackage{amssymb}
\usepackage{amsmath}
\usepackage{mathtools}
\usepackage{booktabs}
\usepackage{float}
\usepackage{verbatim}
\usepackage{tikz-cd}
\usepackage{bm} 
\usepackage[colorlinks,linkcolor=cyan,citecolor=cyan]{hyperref}
\usepackage[hang,flushmargin]{footmisc} 
\usepackage[square,comma,sort&compress,numbers]{natbib} 
\usepackage{mathrsfs} 
\usepackage[font=footnotesize,skip=0pt,textfont=rm,labelfont=rm]{caption,subcaption} 

\usepackage{amsthm}
\newtheorem{theorem}{Theorem}[section]
\newtheorem{proposition}{Proposition}[section]
\newtheorem{lemma}{Lemma}[section]
\newtheorem{definition}{Definition}[section]

\newtheorem{remark}{Remark}[section]

\begin{document}
	\title{Nonclosure of stationary Yang-Mills connections under weak convergence}
	\author[X. Han]{Xiaoli Han}
	\address{Department of Mathematical Sciences, Tsinghua University, Beijing 100084, China}
	\email{hanxiaoli@mail.tsinghua.edu.cn}
	
	\author[J. Li]{Jiayu Li}
	\address{The Department of Mathematics, Nanjing University, Nanjing 210093, China}
	\email{jiayuli@nju.edu.cn}
	
	\author[Y. Wen]{Yang Wen$^{*}$}
	\address{School of Mathematics and Statistics, Nanjing University of Science and Technology, Nanjing 210094, China}
	\email{wenyang2325@njust.edu.cn}
	\thanks{$^{*}$Corresponding author.}
	\begin{abstract}
		We construct a sequence of stationary weak Yang-Mills connections on the trivial vector bundle $B^5\times\mathbb{R}^{15}$ with structure group $\operatorname{SO}(15)$ whose weak limit is a nonstationary weak Yang-Mills connection. Motivated by the stationarity question arising in Tian's work on Yang-Mills compactness, our result establishes the failure of weak closure for stationary weak Yang-Mills connections with an isolated singularity. It also provides a Yang-Mills analogue of the phenomenon exhibited by Ding, Li and Li for harmonic maps.
	\end{abstract}
	\date{}
	\maketitle
	\section{introduction}\ 
	
	Let $(M^n,g)$ be a Riemannian manifold and let $E$ be a Riemannian vector bundle of rank $r$ over $M$ with compact structure group $G\subset SO(r)$. Let $\mathfrak{g}_E$ be the associated adjoint bundle. For a $G$-compatible connection $A$ with curvature $F_A\in L^2(M)$, the Yang-Mills functional is defined by
	\begin{align*}
		\operatorname{YM}(A)=\frac12\int_M|F_A|^2dV_g.
	\end{align*}
	Let $d^A$ be the exterior covariant derivative induced by $A$ and let $\delta^A$ be its formal adjoint. Critical points of the Yang-Mills functional with respect to compactly supported variations of the connection are called Yang-Mills connections and satisfy
	\begin{align}\label{YM equation}
		\delta^AF_A=0.
	\end{align}
	We are interested in weak solutions of \eqref{YM equation} with $W^{1,2}(M)\cap L^4(M)$ regularity, understood relative to a fixed smooth reference connection. A weak Yang-Mills connection $A$ is called stationary if
	\begin{align*}
		\frac{d}{dt}|_{t=0}\operatorname{YM}(\Phi_t^*A)=0
	\end{align*}
	for every flow $\{\Phi_t\}$ generated by a smooth compactly supported vector field on $M$. Here the energy of the pullback connection is computed using the fixed metric $g$ on $M$ and the pullback bundle metric on $\Phi_t^*E$. An equivalent formulation is given in Definition \ref{def:stationary}. Every smooth Yang-Mills connection is stationary, whereas for weak solutions stationarity is an additional condition.
	
	Weak convergence of Yang-Mills connections can be accompanied by concentration of curvature energy. For a sequence of smooth Yang-Mills connections with uniformly bounded energy, let $A$ be a limit obtained after passing to a subsequence and applying suitable gauge transformations. The limiting energy measure has the form
	\begin{align*}
		|F_{A_k}|^2dV_g\rightharpoonup |F_A|^2dV_g+\nu.
	\end{align*}
	Tian \cite{Tian} proved the rectifiability of the defect measure: in dimensions $n>4$, it can be written as $\nu=\theta\mathcal{H}^{n-4}\lfloor S$ for an $(n-4)$-rectifiable set $S$. Rivi\`ere \cite{Riviere} established the energy identity in dimension four and obtained a higher-dimensional result under an additional uniform $L^1$ bound on the covariant Hessian of the curvature. Naber and Valtorta \cite{NV} subsequently proved the higher-dimensional energy identity without that additional assumption. In particular, for $\mathcal{H}^{n-4}$-almost every $x\in S$, there is a finite collection of Yang-Mills bubbles $B_{x,1},\dots,B_{x,N_x}$ on $S^4$ such that
	\begin{align*}
		\theta(x)=\sum_{l=1}^{N_x}\int_{S^4}|F_{B_{x,l}}|^2d\Theta.
	\end{align*}
	The energy identity describes the energy carried by the bubbles. Stationarity imposes a balance law for the full quadratic stress tensor, and the energy identity alone does not imply that this balance law passes to the weak limit. This leads to the question of whether stationarity is preserved under weak convergence.
	
	This question is closely related to the geometry of the defect measure. In his discussion of blow-up loci, Tian \cite[Section 5]{Tian2} related the first variation of the defect measure to the stationarity identity for the limiting connection and expressed doubt that the limit is stationary in general. His discussion concerns limits of smooth Yang-Mills connections. It also motivates the corresponding weak closure question for stationary weak Yang-Mills connections.
	
	An analogous phenomenon is known for harmonic maps. Ding, Li and Li \cite{DLL} constructed a sequence of stationary harmonic maps in dimension three whose weak limit is harmonic but not stationary. Their construction uses radial extensions of harmonic maps from $S^2$ to $S^2$, a condition on the first moment of the spherical energy density, and bubbling that destroys this condition in the limit. More recently, Naber and Valtorta \cite{NV2} established the energy identity for stationary harmonic maps, expressing the density of the defect measure as a sum of energies of harmonic spheres. These results illustrate the distinction between accounting for concentrated energy and preserving stationarity.
	
	In this paper, we construct an explicit counterexample to weak closure for stationary weak Yang-Mills connections in dimension five, the first supercritical dimension for the Yang-Mills functional. Our main result is the following.
	\begin{theorem}\label{thm:main}
		There exist a uniformly bounded sequence of weak Yang-Mills connections $A_k$ in $W^{1,2}(B^5)\cap L^4(B^5)$ on the trivial vector bundle $B^5\times\mathbb{R}^{15}$ with structure group $SO(15)$ and a weak Yang-Mills connection $A_\infty\in W^{1,2}(B^5)\cap L^4(B^5)$ such that
		\begin{align*}
			A_k&\to A_\infty\quad\text{in }L^2(B^5),\\
			A_k&\rightharpoonup A_\infty\quad\text{in }W^{1,2}(B^5),\\
			F_{A_k}&\rightharpoonup F_{A_\infty}\quad\text{in }L^2(B^5).
		\end{align*}
		Each $A_k$ is stationary, but $A_\infty$ is not stationary. Moreover, $A_k$ and $A_\infty$ are smooth on $B^5\backslash\{0\}$. The Sobolev spaces and convergence are understood in the standard trivialization.
	\end{theorem}
	
	Thus stationarity is not preserved under weak $W^{1,2}$ convergence in this class, even when the curvatures converge weakly in $L^2$ to the curvature of the limit. The corresponding question for sequences that are smooth on the whole ball is not settled by this construction.
	
	Our construction follows the radial extension strategy of Ding, Li and Li \cite{DLL}. Let $\pi:B^5\backslash\{0\}\to S^4$ be given by $\pi(y)=\frac{y}{|y|}$. We show that the radial pullback of a smooth Yang-Mills connection $b$ on $S^4$ is a weak Yang-Mills connection on $B^5$, and that
	\begin{align*}
		\pi^*b\text{ is stationary on }B^5\quad\Longleftrightarrow\quad\int_{S^4}x|F_b|^2d\Theta=0.
	\end{align*}
	Starting from an explicit Yang-Mills connection on $S^4\times\mathbb{R}^5$, we use conformal transformations and rotations to produce concentrating families. We then take the direct sum of a fixed connection with two concentrating connections whose energy moments cancel the moment of the fixed connection. The concentrating terms converge weakly to zero, leaving a limit with a nonzero energy moment. A suitable gauge choice allows us to verify the required Sobolev bounds in a fixed trivialization.
	\section{Preliminaries}\ 
	
	We fix our notation and recall the identities for connections and curvature that will be used below. Let $(M,g)$ be an $n$-dimensional Riemannian manifold, and let $D$ denote its Levi-Civita connection. Let $E\to M$ be a Riemannian vector bundle of rank $r$ with structure group $G\subset\operatorname{SO}(r)$, where $G$ is	 compact. We write $\mathfrak{g}$ for the Lie algebra of $G$ and $\mathfrak{g}_E\subset\mathfrak{so}(E)$ for the associated adjoint bundle. We denote the space of smooth $\mathfrak{g}_E$-valued $p$-forms by $\Omega^p(\mathfrak{g}_E)$. The bundle metric on $E$ induces the pointwise inner product
	\begin{align*}
		\langle\phi,\psi\rangle={\rm Tr\,}(\phi^T\psi)
	\end{align*}
	on $\Omega^0(\mathfrak{g}_E)$. Its $\operatorname{Ad}(G)$-invariance implies that
	\begin{align*}
		\langle[\phi,\psi],\rho\rangle=\langle\phi,[\psi,\rho]\rangle
	\end{align*}
	for any $\phi,\psi,\rho\in\Omega^0(\mathfrak{g}_E)$. The induced pointwise inner product on $\Omega^p(\mathfrak{g}_E)$ is given by
	\begin{align*}
		\langle\phi,\psi\rangle=\frac{1}{p!}\sum_{1\leq i_1,\ldots,i_p\leq n}\langle\phi(e_{i_1},\ldots,e_{i_p}),\psi(e_{i_1},\ldots,e_{i_p})\rangle,
	\end{align*}
	where $\{e_i\}_{i=1}^n$ is a local orthonormal frame of $TM$.
	
	Let $\nabla$ be a connection on $E$ compatible with the $G$-structure, and use the same symbol for the induced connections on $\mathfrak{g}_E$ and $\Lambda^pT^*M\otimes\mathfrak{g}_E$. Thus, for $\phi\in\Omega^p(\mathfrak{g}_E)$,
	\begin{align*}
		(\nabla_X\phi)(X_1,\ldots,X_p)=\nabla_X(\phi(X_1,\ldots,X_p))-\sum_{k=1}^p\phi(X_1,\ldots,D_XX_k,\ldots,X_p).
	\end{align*}
	The $W^{1,2}$-norm of 1-forms is given by
	\begin{align*}
		\|\phi\|_{W^{1,2}}^2=\|\phi\|_{L^2}^2+\|\nabla\phi\|_{L^2}^2
	\end{align*}
	for any $\phi\in\Omega^1(\mathfrak{g}_E)$, where
	\begin{align*}
		\|\nabla\phi\|_{L^2}^2=\int_M\sum_{i,j=1}^n|\nabla_{e_i}\phi(e_j)|^2dV.
	\end{align*}
	This connection induces the exterior covariant derivative
	\begin{align*}
		d^\triangledown:\Omega^p(\mathfrak{g}_E)\to\Omega^{p+1}(\mathfrak{g}_E)
	\end{align*}
	and its formal adjoint
	\begin{align*}
		\delta^\triangledown:\Omega^p(\mathfrak{g}_E)\to\Omega^{p-1}(\mathfrak{g}_E),
	\end{align*}
	which are given by
	\begin{align*}
		d^\triangledown\phi(X_1,\ldots,X_{p+1})&=\sum_{i=1}^{p+1}(-1)^{i+1}\nabla_{X_i}\phi(X_1,\ldots,\widehat{X_i},\ldots,X_{p+1}),\\
		\delta^\triangledown\phi(X_1,\ldots,X_{p-1})&=-\sum_{i=1}^n\nabla_{e_i}\phi(e_i,X_1,\ldots,X_{p-1}).
	\end{align*}
	
	In a local orthonormal trivialization of $E$, the connection has the form
	\begin{align*}
		\nabla=d+A.
	\end{align*}
	We write $d^\triangledown=d^A$, $\delta^\triangledown=\delta^A$ for the trivialization. The curvature tensor $F_A$ is
	\begin{align*}
		F_A=dA+A\wedge A,
	\end{align*}
	where
	\begin{align*}
		\phi\wedge\psi(X,Y)=\frac12([\phi(X),\psi(Y)]-[\phi(Y),\psi(X)])
	\end{align*}
	for any $\phi,\psi\in\Omega^1(\mathfrak{g}_E)$ and $X,Y\in\Gamma(TM)$. We use the following definition of stationary Yang-Mills connections; see \cite[(4.5.12) and Remark 11]{Tian}.
	\begin{definition}\label{def:stationary}
		A connection $A$ with $F_A\in L^2_{\operatorname{loc}}(M)$ on a Riemannian manifold $(M^n,g)$ is called stationary if the following conditions hold:
		\begin{enumerate}
			\item[(i)] $A$ is a weak solution of the Yang-Mills equation \eqref{YM equation}, that is,
			\begin{align*}
				\int_M\langle F_A,d^A\varphi\rangle dV_g=0
			\end{align*}
			for every $\varphi\in\Omega^1(\mathfrak{g}_E)$ with compact support.
			\item[(ii)] $A$ satisfies
			\begin{align}
				\int_M\left(|F_A|^2\operatorname{div}_gX-2\sum_{i,j,k=1}^n\langle F_A(e_i,e_k),F_A(e_j,e_k)\rangle g(D_{e_i}X,e_j)\right)dV_g=0
			\end{align}
			for every smooth compactly supported vector field $X$ on $M$, or equivalently,
			\begin{align}
				\operatorname{div}_gT=0
			\end{align}
			in the sense of distributions, where
			\begin{align*}
				T(e_i,e_j)=|F_A|^2\delta_{ij}-2\sum_{k=1}^n\langle F_A(e_i,e_k),F_A(e_j,e_k)\rangle
			\end{align*}
			and $\{e_i\}_{i=1}^n$ is a local orthonormal frame of $TM$.
		\end{enumerate}
	\end{definition}
	\begin{remark}
		With our normalization of the 2-form norm, the stationarity identity above is one half of \cite[(4.5.12)]{Tian}; the resulting notion of stationarity is unchanged.
	\end{remark}
	\section{Radial pullbacks of Yang-Mills connections}
	
	In this section, we study radial pullbacks of smooth Yang-Mills connections on $S^4$ and give a criterion for their stationarity on $B^5$.
	
	Let $\iota:(S^4,g_{S^4})\to(\mathbb{R}^5,g)$ be the standard isometric embedding and $n(x)=x\in NS^4$ be the unit outer normal vector field. Let $d$ and $D$ be the Levi-Civita connections of $T\mathbb{R}^5$ and $TS^4$, respectively. The Gauss formula gives
	\begin{equation}
		\begin{split}
			&d_XY=D_XY-g_{S^4}(X,Y)n,\\
			&d_Xn=X
		\end{split}
	\end{equation}
	for any $X,Y\in\mathscr{X}(S^4)$. Let $E=S^4\times\mathbb{R}^5$ be a trivial bundle over $S^4$ equipped with the structure group $SO(5)$. Then $\Omega^p(\mathfrak{g}_E)$ is identified with $\Omega^p(S^4)\otimes\mathfrak{so}(5)$. We identify a metric connection $\nabla^b=d+b$ on $E$ with its connection form $b\in\Omega^1(\mathfrak{g}_E)$.
	\begin{lemma}\label{lem:F_B(r,theta)=F_b(theta)}
		Let $\pi:B^5\backslash\{0\}\to S^4$ with $\pi(y)=\frac{y}{|y|}$. Assume that $b$ is a smooth Yang-Mills connection over $S^4$. Then $B:=\pi^*b$ is a smooth Yang-Mills connection over $B^5\backslash\{0\}$ satisfying
		\begin{equation}
			\begin{split}
				&F_B(\iota_{r*}X_i,\iota_{r*}X_j)|_{rx}=F_b(X_i,X_j)|_x,\\
				&i_{\frac\partial{\partial r}}F_B=0
			\end{split}
		\end{equation}
		for any $r\in(0,1)$ and $x\in S^4$, where $\iota_r:S^4\to B^5$ is defined by $\iota_r(x)=rx$, and $\{X_i\mid1\le i\le4\}$ and $\{\frac\partial{\partial r},\iota_{r*}X_i\mid1\le i\le4\}$ are bases of $T_xS^4$ and $T_{rx}(B^5\backslash\{0\})$, respectively.
	\end{lemma}
	\begin{proof}
		Choose normal coordinates $(x^1,\dots,x^4)$ near $x_0\in S^4$, and set $X_i=\frac\partial{\partial x^i}$, $T=\frac\partial{\partial r}$, and $Y_i|_{rx}=\iota_{r*}X_i|_x$. Then $g_{S^4}(X_i,X_j)(x_0)=\delta_{ij}$, $D_{X_i}X_j(x_0)=0$, and
		\begin{align*}
			[X_i,X_j]=0,\quad[Y_i,Y_j]=[T,Y_i]=0.
		\end{align*}
		The Euclidean metric has the warped product form
		\begin{align*}
			g=dr^2+r^2g_{S^4}.
		\end{align*}
		Applying Bishop and O'Neill's connection formulas \cite[Lemma 7.3]{BN} with base $(0,1)$, fiber $S^4$, and warping function $f(r)=r$, we obtain
		\begin{align*}
			d_TT&=0,\\
			d_TY_i&=\frac1rY_i,\\
			d_{Y_i}T&=\frac1rY_i,\\
			d_{Y_i}Y_j&=\iota_{r*}(D_{X_i}X_j)-\frac1rg(Y_i,Y_j)T.
		\end{align*}
		Moreover,
		\begin{align*}
			g(Y_i,Y_j)|_{rx}=r^2g_{S^4}(X_i,X_j)|_x.
		\end{align*}
		Thus $\{T,r^{-1}Y_1,\dots,r^{-1}Y_4\}$ is an orthonormal basis at $rx_0$, and $d_{Y_i}Y_j|_{rx_0}=-r\delta_{ij}T|_{rx_0}$.
		
		Since $\pi\circ\iota_r=\operatorname{id}_{S^4}$, we have $\pi_*T=0,\ \pi_*Y_i=X_i$. Using $dB=d\pi^*b=\pi^*db$, we have
		\begin{align*}
			dB(\iota_{r*}X_i,\iota_{r*}X_j)_{rx_0}=\pi^*db(\iota_{r*}X_i,\iota_{r*}X_j)=db(\pi_*\iota_{r*}X_i,\pi_*\iota_{r*}X_j)|_{x_0}=db(X_i,X_j)|_{x_0}
		\end{align*}
		and
		\begin{align*}
			dB(T,\iota_{r*}X_i)=\pi^*db(T,\iota_{r*}X_i)=db(\pi_*T,\pi_*\iota_{r*}X_i)=0
		\end{align*}
		for any $1\le i,j\le4$. We also have
		\begin{align*}
			B\wedge B(\iota_{r*}X_i,\iota_{r*}X_j)=[\pi^*b(\iota_{r*}X_i),\pi^*b(\iota_{r*}X_j)]=[b(\pi_*\iota_{r*}X_i),b(\pi_*\iota_{r*}X_j)]=b\wedge b(X_i,X_j)
		\end{align*}
		and
		\begin{align*}
			B\wedge B(T,\iota_{r*}X_i)=[b(\pi_*T),b(\pi_*\iota_{r*}X_i)]=0
		\end{align*}
		Hence, we have
		\begin{align*}
			&F_B(\iota_{r*}X_i,\iota_{r*}X_j)|_{rx}=F_b(X_i,X_j)|_x,\\
			&i_TF_B=0.
		\end{align*}
		Since $i_{d_{Y_i}Y_j}F_B=0$ at $rx_0$ and $D_{X_i}X_j=0$ at $x_0$, we have
		\begin{align*}
			\delta^BF_B(Y_j)\mid_{rx_0}&=-\nabla^B_TF_B(T,Y_j)-r^{-2}\sum_{i=1}^4\nabla^B_{Y_i}F_B(Y_i,Y_j)=-r^{-2}\sum_{i=1}^4\nabla^B_{Y_i}(F_B(Y_i,Y_j))\\
			&=-r^{-2}\sum_{i=1}^4\nabla^b_{X_i}(F_b(X_i,X_j))\mid_{x_0}=-r^{-2}\sum_{i=1}^4\nabla^b_{X_i}F_b(X_i,X_j)=r^{-2}\delta^bF_b(X_j)\\
			&=0
		\end{align*}
		and
		\begin{align*}
			\delta^BF_B(T)\mid_{rx_0}&=-r^{-2}\sum_{i=1}^4\nabla_{Y_i}^BF_B(Y_i,T)\\
			&=-r^{-2}\sum_{i=1}^4(\nabla^B_{Y_i}(F_B(Y_i,T))-F_B(d_{Y_i}Y_i,T)-F_B(Y_i,d_{Y_i}T))\\
			&=0,
		\end{align*}
		where we use $d_{Y_i}Y_i=-rT$ and $d_{Y_i}T=\frac1rY_i$ at $rx_0$. Thus $B$ is a Yang-Mills connection on $B^5\backslash\{0\}$.
	\end{proof}
	Even if $b$ is a Yang-Mills connection over $S^4$, $\pi^*b$ need not be stationary over $B^5$. In fact, we have the following proposition.
	\begin{proposition}\label{lem:stationary}
		Assume that $b$ is a smooth Yang-Mills connection over $S^4$. Then $B=\pi^*b$ defines a weak Yang-Mills connection over $B^5$, and it is stationary if and only if
		\begin{align*}
			\int_{S^4}x|F_b|^2d\Theta=(\int_{S^4}y^1|F_b|^2d\Theta,\dots, \int_{S^4}y^5|F_b|^2d\Theta)=0,
		\end{align*}
		where $x=(y^1,\dots,y^5)\in S^4$ and $d\Theta$ is the volume form of $S^4$.
	\end{proposition}
	\begin{proof}
		We first show that $B$ is a weak solution of the Yang-Mills equation in $B^5$. The smoothness of $b$ and the homogeneity of $B$ imply that $B\in W^{1,2}(B^5)\cap L^4(B^5)$ and $F_B\in L^2(B^5)$. For any $\epsilon\in(0,\frac12)$, take a cut-off function $\eta(r)$ such that $\eta\equiv0$ on $(2\epsilon,1)$, $\eta\equiv1$ on $[0,\epsilon)$ and $|d\eta|\le\frac{C}\epsilon$. For any $\varphi\in\Omega^1(B^5,\mathfrak{so}(5))$ with compact support, the identity $i_TF_B=0$ gives $\langle F_B,d\eta\wedge\varphi\rangle=0$. Since $B$ is Yang-Mills away from the origin,
		\begin{align*}
			\int_{B^5}(1-\eta)\langle F_B,d^B\varphi\rangle dV=\int_{B^5}\langle F_B,d^B((1-\eta)\varphi)\rangle dV=0.
		\end{align*}
		Letting $\epsilon\to0$, we conclude that $B$ is a weak Yang-Mills connection over $B^5$.
		
		To determine when $B$ is stationary, for any vector field $X\in\mathscr{X}(B^5)$ with compact support in $B^5$ and any $rx\in B^5\backslash\{0\}$, assume $\{e_i\}$ is an orthonormal basis of $T_xS^4$ and let $\tilde e_0=T=\frac\partial{\partial r}$, $\tilde e_i=\frac1r\iota_{r*}e_i$ for $i=1,\dots,4$. Then $\{\tilde e_i\mid0\le i\le4\}$ is an orthonormal basis of $T_{rx}(B^5\backslash\{0\})$. Note that
		\begin{align*}
			&\int_{B^5}|F_B|^2\operatorname{div}(X)-2\sum_{i,j,k=0}^4\langle F_B(\tilde e_i,\tilde e_k),F_B(\tilde e_j,\tilde e_k)\rangle g(d_{\tilde e_i}X,\tilde e_j)dV\\
			=&\int_{B_{2\epsilon}(0)}|F_B|^2\operatorname{div}(\eta X)-2\sum_{i,j,k=0}^4\langle F_B(\tilde e_i,\tilde e_k),F_B(\tilde e_j,\tilde e_k)\rangle g(d_{\tilde e_i}(\eta X),\tilde e_j)dV\\
			&+\int_{B_1(0)\backslash B_\epsilon(0)}|F_B|^2\operatorname{div}((1-\eta)X)-2\sum_{i,j,k=0}^4\langle F_B(\tilde e_i,\tilde e_k),F_B(\tilde e_j,\tilde e_k)\rangle g(d_{\tilde e_i}((1-\eta) X),\tilde e_j)dV.
		\end{align*}
		 Define $X^T=X-g(X,\tilde e_0)\tilde e_0$, the component of $X$ tangent to $\iota_r(S^4)$. Unless otherwise indicated, quantities on $B^5$ and $S^4$ are evaluated at $rx$ and $x$, respectively, and the evaluation points are omitted. Using Lemma \ref{lem:F_B(r,theta)=F_b(theta)}, we have
		\begin{align*}
			&\sum_{i,j,k=0}^4\langle F_B(\tilde e_i,\tilde e_k),F_B(\tilde e_j,\tilde e_k)\rangle g(d_{\tilde e_i}((1-\eta) X),\tilde e_j)\\
			=&\frac{1}{r^4}\sum_{i,j,k=1}^4\langle F_b(e_i,e_k),F_b(e_j,e_k)\rangle g(d_{\tilde e_i}((1-\eta)X),\tilde e_j)\\
			=&\frac{1}{r^4}\sum_{i,j,k=1}^4\langle F_b(e_i,e_k),F_b(e_j,e_k)\rangle(g(d_{\tilde e_i}((1-\eta)X^T),\tilde e_j)+(1-\eta)g(X,\tilde e_0)g(d_{\tilde e_i}\tilde e_0,\tilde e_j))\\
			=&\frac{1}{r^4}\sum_{i,j,k=1}^4\langle F_b(e_i,e_k),F_b(e_j,e_k)\rangle g_{S^4}(D_{e_i}((1-\eta)\pi_*X^T),e_j)+\frac2{r^5}(1-\eta)g(X,\tilde e_0)|F_b|^2,
		\end{align*}
		where we use $g(d_{\tilde e_i}\tilde e_0,\tilde e_j)=\frac1r\delta_{ij}$. Since $d_{\tilde e_0}\tilde e_0=0$ and $d_{\tilde e_i}\tilde e_0=\frac1r\tilde e_i$ for $i=1,\dots,4$, we have
		\begin{align*}
			\operatorname{div}((1-\eta)X)=&\sum_{i=0}^4g(d_{\tilde e_i}((1-\eta)g(X,\tilde e_0)\tilde e_0+(1-\eta)X^T),\tilde e_i)\\
			=&\frac\partial{\partial r}((1-\eta)g(X,\tilde e_0))+\frac4r(1-\eta)g(X,\tilde e_0)+\operatorname{div}_{S^4}((1-\eta)\pi_*X^T).
		\end{align*}
		Therefore,
		\begin{align*}
			|F_B|^2\operatorname{div}((1-\eta)X)=(\frac1{r^4}|F_b|^2\frac{\partial }{\partial r}((1-\eta)g(X,\tilde e_0))+\frac{1}{r^4}|F_b|^2\operatorname{div}_{S^4}((1-\eta)\pi_*X^T)+\frac4{r^5}(1-\eta)g(X,\tilde e_0)|F_b|^2).
		\end{align*}
		Since $b$ is a smooth Yang-Mills connection, it is stationary (see \cite[(2.1.8)]{Tian}). Applying the stationarity identity on $S^4$ for each fixed $r$, we have
		\begin{align*}
			&\int_{B_1\backslash B_\epsilon}\frac{1}{r^4}|F_b|^2\operatorname{div}_{S^4}((1-\eta)\pi_*X^T)-2\frac{1}{r^4}\sum_{i,j,k=1}^4\langle F_b(e_i,e_k),F_b(e_j,e_k)\rangle g_{S^4}(D_{e_i}((1-\eta)\pi_*X^T),e_j)dV\\
			=&\int_\epsilon^1\int_{S^4}|F_b|^2\operatorname{div}_{S^4}((1-\eta)\pi_*X^T)-2\sum_{i,j,k=1}^4\langle F_b(e_i,e_k),F_b(e_j,e_k)\rangle g_{S^4}(D_{e_i}((1-\eta)\pi_*X^T),e_j)d\Theta dr\\
			=&0.
		\end{align*}
		Using Stokes' formula, we have
		\begin{align*}
			&\int_{B_1\backslash B_\epsilon}\frac1{r^4}|F_b|^2\frac{\partial }{\partial r}((1-\eta)g(X,\tilde e_0))dV\\
			=&\int_{S^4}|F_b|^2\int_\epsilon^1\frac{\partial }{\partial r}((1-\eta)g(X,\tilde e_0))drd\Theta\\
			=&0.
		\end{align*}
		Hence we obtain
		\begin{align*}
			\int_{B_1(0)\backslash B_\epsilon(0)}|F_B|^2\operatorname{div}((1-\eta)X)-2\sum_{i,j,k=0}^4\langle F_B(\tilde e_i,\tilde e_k),F_B(\tilde e_j,\tilde e_k)\rangle g(d_{\tilde e_i}((1-\eta) X),\tilde e_j)dV=0.
		\end{align*}
		Since
		\begin{align*}
			\operatorname{div}(\eta X)=\frac{\partial}{\partial r}(\eta g(X,\tilde e_0))+\eta\sum_{i=1}^4g(d_{\tilde e_i}X,\tilde e_i),
		\end{align*}
		we have
		\begin{align*}
			&\int_{B_{2\epsilon}(0)}|F_B|^2\operatorname{div}(\eta X)-2\sum_{i,j,k=0}^4\langle F_B(\tilde e_i,\tilde e_k),F_B(\tilde e_j,\tilde e_k)\rangle g(d_{\tilde e_i}(\eta X),\tilde e_j)dV\\
			=&\int_{B_{2\epsilon}(0)}|F_B|^2(\frac{\partial}{\partial r}(\eta g(X,\tilde e_0))+\eta \sum_{i=1}^4g(d_{\tilde e_i}X,\tilde e_i))-2\eta\sum_{i,j,k=1}^4\langle F_B(\tilde e_i,\tilde e_k),F_B(\tilde e_j,\tilde e_k)\rangle g(d_{\tilde e_i}X,\tilde e_j)dV\\
			=&\int_{S^4}|F_b|^2\int_0^{2\epsilon}(\frac{\partial}{\partial r}(\eta g(X,\tilde e_0))+\eta \sum_{i=1}^4g(d_{\tilde e_i}X,\tilde e_i))drd\Theta\\
			&-2\int_{S^4}\sum_{i,j,k=1}^4\langle F_b(e_i,e_k),F_b(e_j,e_k)\rangle\int_0^{2\epsilon}\eta g(d_{\tilde e_i}X,\tilde e_j)drd\Theta\\
			=&-\lim_{r\to0}\int_{S^4}|F_b|^2g(X,\tilde e_0)|_{rx}d\Theta+O(\epsilon).
		\end{align*}
		Letting $\epsilon\to0$ and noting that $\lim_{r\to0}g(X,\tilde e_0)|_{rx}=\sum_{i=1}^5y^ig(X,\frac\partial{\partial y^i})|_0$, we obtain
		\begin{align*}
			\int_{B^5}|F_B|^2\operatorname{div}(X)-2\sum_{i,j,k=0}^4\langle F_B(\tilde e_i,\tilde e_k),F_B(\tilde e_j,\tilde e_k)\rangle g(d_{\tilde e_i}X,\tilde e_j)dV=-\sum_{i=1}^5g(X,\frac\partial{\partial y^i})\mid_0\int_{S^4}y^i|F_b|^2d\Theta.
		\end{align*}
		Since $g(X,\frac\partial{\partial y^i})|_0$ is arbitrary, we can see that $B$ is stationary if and only if
		\begin{align*}
			\int_{S^4}y^i|F_b|^2d\Theta=0\quad\text{for any }i=1,\dots,5.
		\end{align*}
	\end{proof}
	\begin{remark}
		The preceding lemma and proposition remain valid for $S^4\times\mathbb{R}^m$ with structure group $SO(m)$, since their proofs are independent of the rank.
	\end{remark}
	Let $\{e_i\mid1\le i\le4\}$ and $\{T=\frac\partial{\partial r},\tilde e_i=\frac1r\iota_{r*}e_i\mid1\le i\le4\}$ be local orthonormal frames of $TS^4$ and $T(B^5\backslash\{0\})$, respectively. Then $\pi_*T=0$ and $\pi_*\tilde e_i=\frac1re_i$. For any $b\in\Omega^1(\mathfrak{g}_E)$ and $p=1,2$, we have
	\begin{align}\label{pi^*b L^p}
		\int_{B^5}|\pi^*b|^{2p}dx=\int_{B^5}\frac1{r^{2p}}|b|^{2p}(\frac{x}{|x|})dx=\frac1{5-2p}\int_{S^4}|b|^{2p}d\Theta.
	\end{align}
	For any $x_1\in S^4$ and $r\in(0,1)$, choose normal coordinates centered at $x_1$ and let $X_i,T,Y_i$ be the corresponding local vector fields defined in Lemma \ref{lem:F_B(r,theta)=F_b(theta)}. In the Sobolev estimates below, $\nabla$ denotes the connection induced by the flat connection in the standard trivialization and the Levi-Civita connection of the relevant base metric. Then at $rx_1$, we have
	\begin{align*}
		|\nabla\pi^*b|^2&=\sum_{i=1}^4(\frac1{r^2}|\nabla_T\pi^*b(Y_i)|^2+\frac1{r^2}|\nabla_{Y_i}\pi^*b(T)|^2+\frac1{r^4}\sum_{j=1}^4|\nabla_{Y_j}\pi^*b(Y_i)|^2)\\
		&=\sum_{i=1}^4(\frac2{r^4}|\pi^*b(Y_i)|^2+\frac1{r^4}\sum_{j=1}^4|\nabla_{\pi_*Y_j}b(\pi_*Y_i)|^2)\\
		&=\frac2{r^4}|b|^2(\frac{x_1}{|x_1|})+\frac1{r^4}|\nabla b|^2(\frac{x_1}{|x_1|}),
	\end{align*}
	where we use
	\begin{align*}
		&\nabla_T\pi^*b(Y_i)=\frac\partial{\partial r}(\pi^*b(Y_i))-\pi^*b(d_TY_i)=-\frac1r\pi^*b(Y_i),\\
		&\nabla_{Y_i}\pi^*b(T)=Y_i(\pi^*b(T))-\pi^*b(d_{Y_i}T)=-\frac1r\pi^*b(Y_i),\\
		&\nabla_{Y_j}\pi^*b(Y_i)=Y_j(\pi^*b(Y_i))-\pi^*b(d_{Y_j}Y_i)=Y_j(\pi^*b(Y_i))=(\pi_*Y_j)(b(\pi_*Y_i))|_{x_1}=\nabla_{\pi_*Y_j}b(\pi_*Y_i)|_{x_1}.
	\end{align*}
	Thus
	\begin{align}\label{|db|_{L^2}}
		\int_{B^5}|\nabla\pi^*b|^2dx=\int_{B^5}\frac2{r^4}|b|^2(\frac{x}{|x|})+\frac1{r^4}|\nabla b|^2(\frac x{|x|})dx=\int_{S^4}2|b|^2+|\nabla b|^2d\Theta.
	\end{align}
	The preceding identities show that radial pullback is a bounded linear map from $W^{1,2}(S^4)$ to $W^{1,2}(B^5)$. 
	\section{The construction of stationary Yang-Mills connections with a non-stationary limit}
	In this section, we construct a sequence of stationary Yang-Mills connections over $B^5$ whose weak limit is not stationary.
	\begin{lemma}
		Let $n$ be the unit normal vector field of $S^4$. Define
		\begin{align}
			a=ndn^T-dn\cdot n^T\in\Omega^1(\mathfrak{g}_E).
		\end{align}
		Then the curvature tensor $F_a=da+a\wedge a$ of $a$ has the form
		\begin{align*}
			F_a=dn\wedge dn^T
		\end{align*}
		with
		\begin{align}\label{|F_a|^2}
			|F_a|^2=12.
		\end{align}
		Moreover, $a$ is a Yang-Mills connection.
	\end{lemma}
	\begin{proof}
		It is easy to verify that
		\begin{align*}
			da=2dn\wedge dn^T.
		\end{align*}
		Using $n^Tdn=dn^T\cdot n=0$ and $n^Tn=1$, we have
		\begin{align*}
			&a(X)a(Y)\\
			=&(n\cdot d_Xn^T-d_Xn\cdot n^T)(n\cdot d_Yn^T-d_Yn\cdot n^T)\\
			=&-n\cdot d_Xn^T\cdot d_Yn\cdot n^T-d_Xn\cdot d_Yn^T
		\end{align*}
		and thus
		\begin{align*}
			a\wedge a(X,Y)=a(X)a(Y)-a(Y)a(X)=-d_Xn\cdot d_Yn^T+d_Yn\cdot d_Xn^T=-dn\wedge dn^T(X,Y).
		\end{align*}
		Hence we obtain $F_a=da+a\wedge a=dn\wedge dn^T$.
		
		For any $x\in S^4$ and $X,Y\in T_xS^4$, regarding $X,Y$ as vectors in $\mathbb{R}^5$ and using $d_Xn=X$, $d_Yn=Y$, we have $F_a(X,Y)=XY^T-YX^T$. Therefore, letting $\{e_i\}_{i=1}^4\subset\mathbb{R}^5$ be an orthonormal basis of $T_xS^4$, we have
		\begin{align*}
			|F_a|^2(x)=\frac12\sum_{i,j=1}^4|F_a(e_i,e_j)|^2=-\frac12\sum_{i,j=1}^4tr((e_ie_j^T-e_je_i^T)^2)=12.
		\end{align*}
		
		For any $x_0\in S^4$, take $\{e_1,\dots,e_4\}$ to be a local orthonormal frame near $x_0$ with $De_i(x_0)=0$. Then at $x_0$, for any $1\le j\le4$, we have
		\begin{align*}
			\sum_{i=1}^4\nabla^a_{e_i}F_a(e_i,e_j)=\sum_{i=1}^4\nabla^a_{e_i}(F_a(e_i,e_j))=\sum_{i=1}^4\nabla^a_{e_i}(e_ie_j^T-e_je_i^T)=\sum_{i=1}^4d_{e_i}(e_ie_j^T-e_je_i^T)+[a(e_i),(e_ie_j^T-e_je_i^T)],
		\end{align*}
		where by using $d_{e_i}e_j=-\delta_{ij}n$ at $x_0$, we have
		\begin{align*}
			\sum_{i=1}^4d_{e_i}(e_ie_j^T-e_je_i^T)=\sum_{i=1}^4-ne_j^T-\delta_{ij}e_in^T+\delta_{ij}ne_i^T+e_jn^T=3(e_jn^T-ne_j^T)
		\end{align*}
		and by direct calculation, we have
		\begin{align*}
			\sum_{i=1}^4[a(e_i),(e_ie_j^T-e_je_i^T)]=\sum_{i=1}^4ne_j^T-\delta_{ij}ne_i^T+\delta_{ij}e_in^T-e_jn^T=3(ne_j^T-e_jn^T).
		\end{align*}
		Therefore, we obtain
		\begin{align*}
			\delta^aF_a(e_j)=-\sum_{i=1}^4\nabla^a_{e_i}F_a(e_i,e_j)=-\sum_{i=1}^4d_{e_i}(e_ie_j^T-e_je_i^T)+[a(e_i),(e_ie_j^T-e_je_i^T)]=0.
		\end{align*}
	\end{proof}
	Let $x_0=(0,0,0,0,-1)\in S^4$. Choose local coordinates $(x^1,\dots,x^4)$ centered at $x_0$ and write $a=\sum_{i=1}^4a_i(x)dx^i$ near $x_0$. Take a cut-off function $\chi\in C^\infty(S^4)$ supported in this coordinate neighborhood, with $\chi\equiv1$ near $x_0$ and $-x_0\notin\operatorname{spt}(\chi)$. Define a gauge transformation
	\begin{align*}
		g(x)=\left\{
		\begin{array}{ll}
			\exp(-\chi(x)\sum_{i=1}^4x^ia_i(x)),&\operatorname{spt}(\chi),\\
			I,&\operatorname{spt}(\chi)^c.
		\end{array}
		\right.
	\end{align*}
	Define $\hat a=a^g=g^{-1}dg+g^{-1}ag$. Then $\hat a$ is a Yang-Mills connection. Since $g(x_0)=I$ and $dg(x_0)=-a(x_0)$, we have $\hat a(x_0)=0$. Moreover, $g=I$ near $-x_0$, so $\hat a(-x_0)=a(-x_0)$.
	
	Consider the inverse stereographic projection
	\begin{align*}
		\phi:\mathbb{R}^4\to S^4\backslash\{x_0\},\quad x\mapsto(\frac{2x}{1+|x|^2},\frac{1-|x|^2}{1+|x|^2}).
	\end{align*}
	For any $\lambda>0$, define a diffeomorphism $f_\lambda:S^4\to S^4$ by $f_\lambda(x_0)=x_0$ and
	\begin{align*}
		f_\lambda(\phi(x))=\phi(\frac{x}\lambda).
	\end{align*}
	Let $(x^1,\dots,x^4)$ be the Euclidean coordinates on $\mathbb{R}^4$. By direct calculation, we have
	\begin{align*}
		\phi^*(f_\lambda^*g_{S^4})(\frac\partial{\partial x^i},\frac\partial{\partial x^j})=\frac{4\lambda^2}{(\lambda^2+|x|^2)^2}\delta_{ij}=\frac{\lambda^2(1+|x|^2)^2}{(\lambda^2+|x|^2)^2}\phi^*g_{S^4}(\frac\partial{\partial x^i},\frac\partial{\partial x^j}).
	\end{align*}
	Therefore, $f_\lambda^*g_{S^4}$ is conformal to $g_{S^4}$. Let
	\begin{align}
		b'_\lambda=f_\lambda^*\hat a.
	\end{align}
	Then we have $b'_1=\hat a$ and $F_{b'_\lambda}=f_\lambda^*F_{\hat a}$. Let $*=*_{g_{S^4}}$, $\tilde*=*_{f_\lambda^*g_{S^4}}$, and let $\tilde\delta^{b'_\lambda}$ denote the formal adjoint of $d^{b'_\lambda}$ with respect to $f_\lambda^*g_{S^4}$. By naturality under pullback,
	\begin{align*}
		\tilde\delta^{b'_\lambda}F_{b'_\lambda}=f_\lambda^*(\delta^{\hat a}F_{\hat a})=0.
	\end{align*}
	Since the Hodge operator satisfies
	\begin{align*}
		*_{e^{2\varphi}g}\psi=e^{(4-2k)\varphi}*_g\psi
	\end{align*}
	for any $\varphi\in C^\infty(S^4)$ and $\psi\in\Omega^k(\mathfrak{g}_E)$, we have
	\begin{align*}
		\tilde\delta^{b'_\lambda}F_{b'_\lambda}=-\tilde*\,d^{b'_\lambda}\tilde*F_{b'_\lambda}=-\frac{(\lambda^2+|x|^2)^2}{\lambda^2(1+|x|^2)^2}*d^{b'_\lambda}*F_{b'_\lambda}=\frac{(\lambda^2+|x|^2)^2}{\lambda^2(1+|x|^2)^2}\delta^{b'_\lambda}F_{b'_\lambda}.
	\end{align*}
	Therefore, $\delta^{b'_\lambda}F_{b'_\lambda}=0$ on $S^4\backslash\{x_0\}$. By smoothness, this also holds at $x_0$. Thus $b'_\lambda$ is a Yang-Mills connection on $(S^4,g_{S^4})$. Let
	\begin{align*}
		\{e_i'=\frac{\lambda^2+|x|^2}{2\lambda}f_{\lambda*}\phi_*(\frac\partial{\partial x^i})\mid1\le i\le4\},\quad\{e_i=\frac{1+|x|^2}2\phi_*(\frac\partial{\partial x^i})\mid1\le i\le4\}
	\end{align*}
	be orthonormal bases at $f_\lambda\circ\phi(x)$ and $\phi(x)$, respectively. Then we have
	\begin{align*}
		|F_{b'_\lambda}|^2(\phi(x))&=\frac12\sum_{i,j=1}^4|F_{b'_\lambda}(e_i,e_j)|^2=\frac12\sum_{i,j=1}^4|F_{\hat a}(f_{\lambda*}e_i,f_{\lambda*}e_j)|^2\\
		&=\frac{\lambda^4(1+|x|^2)^4}{2(\lambda^2+|x|^2)^4}\sum_{i,j=1}^4|F_{\hat a}(e_i',e_j')|^2=\frac{\lambda^4(1+|x|^2)^4}{(\lambda^2+|x|^2)^4}|F_{\hat a}|^2(f_\lambda\circ\phi(x))\\
		&=\frac{12\lambda^4(1+|x|^2)^4}{(\lambda^2+|x|^2)^4},
	\end{align*}
	where we use \eqref{|F_a|^2} and $|F_{\hat a}|=|F_a|$. In stereographic coordinates, the volume form is given by
	\begin{align*}
		d\Theta=\frac{16}{(1+|x|^2)^4}dx.
	\end{align*}
	Therefore, we have
	\begin{align*}
		\int_{S^4}|F_{b'_\lambda}|^2d\Theta=\int_{\mathbb{R}^4}\frac{192\lambda^4}{(\lambda^2+|x|^2)^4}dx=192\lambda^4|S^3|\int_0^\infty\frac{r^3}{(\lambda^2+r^2)^4}dr.
	\end{align*}
	Substituting $s=\frac{r^2}{\lambda^2}$ into the above equation, we have
	\begin{align*}
		\int_0^\infty\frac{r^3}{(\lambda^2+r^2)^4}dr=\frac1{2\lambda^4}\int_0^\infty\frac{s}{(s+1)^4}ds.
	\end{align*}
	Thus
	\begin{align}\label{YM(b'_lambda;S^4)}
		\int_{S^4}|F_{b'_\lambda}|^2d\Theta=96|S^3|\int_0^\infty\frac{s}{(s+1)^4}ds:=C_0.
	\end{align}
	Using a similar argument, we obtain
	\begin{equation}
		\begin{split}
			&\int_{S^4}x|F_{b'_\lambda}|^2d\Theta\\
			=&192\int_{\mathbb{R}^4}\frac{\lambda^4\phi(x)}{(\lambda^2+|x|^2)^4}dx\\
			=&-192\lambda^4\int_{\mathbb{R}^4}\frac{1-|x|^2}{(\lambda^2+|x|^2)^4(1+|x|^2)}dx\cdot x_0\\
			=&-12C_0\lambda^4\int_0^\infty\frac{(1-t^2)t^3}{(\lambda^2+t^2)^4(1+t^2)}dt\cdot x_0.
		\end{split}
	\end{equation}
	Define
	\begin{align*}
		\alpha(\lambda)=12\lambda^4\int_0^\infty\frac{(1-t^2)t^3}{(\lambda^2+t^2)^4(1+t^2)}dt=6\int_0^\infty\frac{(1-\lambda^2s)s}{(1+s)^4(1+\lambda^2s)}ds.
	\end{align*}
	By dominated convergence and differentiation under the integral sign, we have
	\begin{align*}
		\lim_{\lambda\to0}\alpha(\lambda)=1,\qquad\alpha(1)=0,\qquad\alpha'(\lambda)=-24\lambda\int_0^\infty\frac{s^2}{(1+s)^4(1+\lambda^2s)^2}ds<0.
	\end{align*}
	In particular, $0<\alpha(\lambda)<1$ for $0<\lambda<1$. We have
	\begin{align}\label{int_S^4theta|F_b_lambda|^2dTheta and YM(b'_lambda;S^4)}
		\int_{S^4}x|F_{b'_\lambda}|^2d\Theta=-\alpha(\lambda)\int_{S^4}|F_{b'_\lambda}|^2d\Theta\cdot x_0.
	\end{align}
	For any $p\in S^4$, take $P\in SO(5)$ such that $Px_0=-p$. Then $b_\lambda^p=(P^{-1})^*b'_\lambda$ is a Yang-Mills connection and satisfies
	\begin{align*}
		\int_{S^4}x|F_{b^p_\lambda}|^2d\Theta=\alpha(\lambda)\int_{S^4}|F_{b^p_\lambda}|^2d\Theta\cdot p=C_0\alpha(\lambda)\cdot p.
	\end{align*}
	\begin{lemma}\label{lem:F_b to0 in L^2}
		For any $p\in S^4$, we have
		\begin{align*}
			&|F_{b_\lambda^p}|^2d\Theta\rightharpoonup C_0\delta_p\textrm{ in the sense of distributions,}\\
			&F_{b'_\lambda}\rightharpoonup0\ \text{in }\ L^2(S^4)
		\end{align*}
		as $\lambda\to0$.
	\end{lemma}
	\begin{proof}
		Due to the symmetry of $S^4$, it suffices to prove the first assertion for $b'_\lambda$, whose concentration point is $-x_0$. For any $h\in C^\infty(S^4)$, substituting $s=\frac{t^2}{\lambda^2}$, we have
		\begin{align*}
			&\int_{S^4}h(x)|F_{b'_\lambda}|^2d\Theta=192\lambda^4\int_{\mathbb{R}^4}\frac{h(\phi(x))}{(\lambda^2+|x|^2)^4}dx=192\lambda^4\int_0^\infty\frac{t^3}{(\lambda^2+t^2)^4}\int_{S^3}h(\phi(t\theta))d\Theta_{S^3}dt\\
			=&96\int_0^\infty\frac{s}{(s+1)^4}\int_{S^3}h(\phi(\lambda\sqrt s\theta))dV_{g_{S^3}}ds.
		\end{align*}
		Since $\phi(\lambda\sqrt{s}\theta)\to-x_0$ for each fixed $s$ and $\theta$, dominated convergence gives
		\begin{align*}
			\lim_{\lambda\to0}\int_{S^4}h(\theta)|F_{b'_\lambda}|^2d\Theta=96|S^3|h(-x_0)\int_0^\infty\frac{s}{(s+1)^4}ds=C_0h(-x_0).
		\end{align*}
		Thus $|F_{b'_\lambda}|^2d\Theta\rightharpoonup C_0\delta_{-x_0}$. Since $P(-x_0)=p$, it follows that $|F_{b_\lambda^p}|^2d\Theta\rightharpoonup C_0\delta_p$.
		
		For any $\psi\in\Omega^2(\mathfrak{g}_E)$, we have
		\begin{align*}
			|\int_{S^4}\langle F_{b'_\lambda},\psi\rangle dV|&\le\|\psi\|_{L^\infty(S^4)}\int_{S^4}|F_{b'_\lambda}|dV\\
			&=\|\psi\|_{L^\infty(S^4)}\int_{\mathbb{R}^4}\frac{32\sqrt3\lambda^2}{(1+|x|^2)^2(\lambda^2+|x|^2)^2}dx\\
			&=16\sqrt3\|\psi\|_{L^\infty(S^4)}|S^3|\lambda^2\int_0^\infty\frac{s}{(1+\lambda^2s)^2(1+s)^2}ds,
		\end{align*}
		where
		\begin{align*}
			\int_0^\infty\frac{s}{(1+\lambda^2s)^2(1+s)^2}ds\le\int_0^1sds+\int_1^{\lambda^{-2}}\frac1sds+\int_{\lambda^{-2}}^\infty\lambda^{-4}\frac1{s^3}ds\le C(|\ln\lambda|+1).
		\end{align*}
		Let $\psi$ be a $L^2$-section in $\Lambda^2T^*S^4\otimes\mathfrak{g}_E$. Choose $\psi_k\in\Omega^2(\mathfrak{g}_E)$ such that $\psi_k\to\psi$ in $L^2$ norm. Then we have
		\begin{align*}
			&|\int_{S^4}\langle F_{b'_\lambda},\psi\rangle dV|\\
			\le&|\int_{S^4}\langle F_{b'_\lambda},\psi-\psi_k\rangle dV|+|\int_{S^4}\langle F_{b'_\lambda},\psi_k\rangle dV|\\
			\le&\|F_{b'_\lambda}\|_{L^2(S^4)}\|\psi-\psi_k\|_{L^2(S^4)}+C\|\psi_k\|_{L^\infty(S^4)}\lambda^2(|\ln\lambda|+1).
		\end{align*}
		Using $\|F_{b'_\lambda}\|_{L^2(S^4)}^2=C_0$, we have
		\begin{align*}
			\limsup_{\lambda\to0}|\int_{S^4}\langle F_{b'_\lambda},\psi\rangle dV|\le\limsup_{k\to\infty}\lim_{\lambda\to0}\sqrt{C_0}\|\psi-\psi_k\|_{L^2(S^4)}+C\|\psi_k\|_{L^\infty(S^4)}\lambda^2(|\ln\lambda|+1)=0.
		\end{align*}
		Therefore, we obtain $F_{b'_\lambda}\rightharpoonup0$.
	\end{proof}
	Let
	\begin{align}
		B=\phi^*\hat a
	\end{align}
	be a connection on $\phi^*E$ over $\mathbb{R}^4$. Since
	\begin{align*}
		|\phi(x)-x_0|=O(\frac1{|x|}),\quad|\phi_*(\frac\partial{\partial x^i})|_{g_{S^4}}=O(\frac1{|x|^2})
	\end{align*}
	as $x\to\infty$ and $|\hat a(e_i)|(p)=O(|p-x_0|)$ for any unit vector $e_i\in T_pS^4$, we have
	\begin{align*}
		|B(\frac{\partial}{\partial x^i})|(x)=|\hat a(\phi_*(\frac\partial{\partial x^i}))|(\phi(x))=O(\frac1{|x|^3})\textrm{ as }x\to\infty.
	\end{align*}
	Similarly, since $|\nabla\hat a(e_i)|(p)=O(1)$, we have
	\begin{align*}
		|\nabla B(\frac\partial{\partial x^i},\frac\partial{\partial x^j})|(x)=O(\frac1{|x|^4})\textrm{ as }x\to\infty.
	\end{align*}
	Therefore,
	\begin{align*}
		|B|(x)=O(\frac1{|x|^3}),\quad|\nabla B|(x)=O(\frac1{|x|^4})\textrm{ as }x\to\infty.
	\end{align*}
	Since $B$ is smooth on $\mathbb{R}^4$, these decay estimates imply that $B\in L^2(\mathbb{R}^4)\cap L^4(\mathbb{R}^4)$ and $\nabla B\in L^2(\mathbb{R}^4)$. Define $\rho_r:\mathbb{R}^4\to\mathbb{R}^4$ by $\rho_r(x)=rx$ for $r>0$. Then $f_\lambda\circ\phi=\phi\circ\rho_{\frac1\lambda}$. Let
	\begin{align}
		B_\lambda=\phi^*b'_\lambda.
	\end{align}
	Then we have
	\begin{align*}
		B_\lambda=\phi^*b'_\lambda=\phi^*f_\lambda^*\hat a=\rho_{\frac1\lambda}^*\phi^*\hat a=\rho_{\frac1\lambda}^*B
	\end{align*}
	and thus
	\begin{align*}
		B_\lambda(x)=\frac1\lambda B(\frac x\lambda),\quad \nabla B_\lambda(x)=\frac1{\lambda^2}\nabla B(\frac x\lambda).
	\end{align*}
	Therefore,
	\begin{align*}
		\|B_\lambda\|_{L^2(\mathbb{R}^4)}=\lambda\|B\|_{L^2(\mathbb{R}^4)},\quad\|B_\lambda\|_{L^4(\mathbb{R}^4)}=\|B\|_{L^4(\mathbb{R}^4)},\quad\|\nabla B_\lambda\|_{L^2(\mathbb{R}^4)}=\|\nabla B\|_{L^2(\mathbb{R}^4)}.
	\end{align*}
	\begin{lemma}\label{lem:b'_lambda convergence in L^2 and W^1,2}
		We have
		\begin{align*}
			b'_\lambda\to0,\ in\ L^2(S^4),\quad b'_\lambda\rightharpoonup0,\ \text{in }\ W^{1,2}(S^4)
		\end{align*}
		as $\lambda\to0$.
	\end{lemma}
	\begin{proof}
		For any $p=2,4$, we have
		\begin{align*}
			&\int_{S^4}|b'_\lambda|^pd\Theta=\int_{S^4}(\sum_{i=1}^4|b_\lambda'(\frac{1+|x|^2}2\phi_*\frac\partial{\partial x^i})|^2)^{\frac p2}d\Theta=\int_{\mathbb{R}^4}(\sum_{i=1}^4(\frac{1+|x|^2}2)^{p-4}|b'_\lambda(\phi_*\frac\partial{\partial x^i})|^2)^{\frac p2}(\phi(x))dx\\
			=&\int_{\mathbb{R}^4}(\frac{1+|x|^2}2)^{p-4}|B_\lambda|^pdx
		\end{align*}
		and thus
		\begin{align}\label{|b_lambda|_L^2}
			\|b'_\lambda\|_{L^2(S^4)}\le2\lambda\|B\|_{L^2(\mathbb{R}^4)},\quad\|b'_\lambda\|_{L^4(S^4)}\le\|B\|_{L^4(\mathbb{R}^4)}.
		\end{align}
		Therefore, 
		\begin{align*}
			b'_\lambda\to0,\ \text{in }L^2(S^4).
		\end{align*}
		In stereographic coordinates, the spherical metric is $\phi^*g_{S^4}=4(1+|x|^2)^{-2}\sum_{i=1}^4(dx^i)^2$. Its Levi-Civita connection satisfies 
		\begin{align*}
			D_{\phi_*\frac\partial{\partial x^i}}(\phi_*\frac\partial{\partial x^j})=-\frac{2}{1+|x|^2}(x^i\phi_*\frac\partial{\partial x^j}+x^j\phi_*\frac\partial{\partial x^i}-\delta_{ij}\sum_{k=1}^4x^k\phi_*\frac\partial{\partial x^k}).
		\end{align*}
		Since $b'_\lambda(\phi_*X)=B_\lambda(X)\circ\phi^{-1}$, we have
		\begin{align*}
			|\nabla_{\phi_*\frac\partial{\partial x^i}}b'_\lambda(\phi_*\frac\partial{\partial x^j})|=|\frac\partial{\partial x^i}(B_\lambda(\frac\partial{\partial x^j}))\circ\phi^{-1}-b'_\lambda(D_{\phi_*\frac\partial{\partial x^i}}\phi_*\frac\partial{\partial x^j})|\le C(|\nabla B_\lambda|+\frac{|x|}{1+|x|^2}|B_\lambda|)
		\end{align*}
		and thus
		\begin{align*}
			&\int_{S^4}|\nabla b'_\lambda|^2d\Theta\\
			=&\sum_{i,j=1}^4\int_{S^4}(\frac{1+|x|^2}2)^4|\nabla_{\phi_*\frac\partial{\partial x^i}}b'_\lambda(\phi_*\frac\partial{\partial x^j})|^2d\Theta\\
			=&\sum_{i,j=1}^4\int_{\mathbb{R}^4}|\nabla_{\phi_*\frac\partial{\partial x^i}}b'_\lambda(\phi_*\frac\partial{\partial x^j})|^2(\phi(x))dx\\
			\le&C(\int_{\mathbb{R}^4}|\nabla B_\lambda|^2+\frac{|x|^2}{(1+|x|^2)^2}|B_\lambda|^2dx).
		\end{align*}
		Therefore, we have
		\begin{align*}
			\|\nabla b'_\lambda\|_{L^2(S^4)}\le C(\|\nabla B\|_{L^2(\mathbb{R}^4)}+\lambda\|B\|_{L^2(\mathbb{R}^4)}).
		\end{align*}
		For any $h\in L^2(T^*S^4\otimes T^*S^4\otimes\mathfrak{g}_E)$, take smooth sections $h_k$ with $h_k\to h$ in $L^2$. Integration by parts gives
		\begin{align*}
			|\int_{S^4}\langle\nabla b'_\lambda,h\rangle d\Theta|&\le C(\|b'_\lambda\|_{L^2}\|\nabla h_k\|_{L^2}+\|\nabla b'_\lambda\|_{L^2}\|h-h_k\|_{L^2})\\
			&\le C(\|b'_\lambda\|_{L^2}\|\nabla h_k\|_{L^2}+\|h-h_k\|_{L^2})
 		\end{align*}
		for $0<\lambda\le1$. For each fixed $k$, letting $\lambda\to0$ gives
		\begin{align*}
			\limsup_{\lambda\to0}|\int_{S^4}\langle\nabla b'_\lambda,h\rangle d\Theta|\le C\|h-h_k\|_{L^2}.
		\end{align*}
		Letting $k\to\infty$, we obtain $\nabla b'_\lambda\rightharpoonup0$ in $L^2(S^4)$.
	\end{proof}
	Applying \eqref{pi^*b L^p}, \eqref{|db|_{L^2}} and Lemma \ref{lem:b'_lambda convergence in L^2 and W^1,2}, we have
	\begin{align*}
		\pi^*b'_\lambda\to0,\ \text{in }L^2(B^5),\quad \pi^*b'_\lambda\rightharpoonup0,\ in\ W^{1,2}(B^5).
	\end{align*}
	For any $\psi\in \Omega^2(\mathfrak{g}_E)$, we have $\|\pi^*\psi\|_{L^2(B^5)}=\|\psi\|_{L^2(S^4)}$. Since $F_{\pi^*b'_\lambda}=\pi^*F_{b'_\lambda}$, Lemma \ref{lem:F_b to0 in L^2} gives
	\begin{align*}
		F_{\pi^*b'_\lambda}\rightharpoonup0\ \text{in }L^2(B^5).
	\end{align*}
	\begin{proof}[Proof of Theorem \ref{thm:main}]
	Now let $E'=S^4\times(\mathbb{R}^5\oplus\mathbb{R}^5\oplus\mathbb{R}^5)$ be the trivial vector bundle over $S^4$ with structure group $SO(15)$. Set $\alpha_0=\alpha(\frac12)$, $\lambda_k=\frac1{k+2}$ for $k\ge1$, and $e_1=(1,0,0,0,0)$. Define
	\begin{align*}
		p_k^\pm=\pm\sqrt{1-\frac{\alpha_0^2}{4\alpha(\lambda_k)^2}}e_1+\frac{\alpha_0}{2\alpha(\lambda_k)}x_0.
	\end{align*}
	Since $0<\alpha_0\le\alpha(\lambda_k)$, these points are well-defined and satisfy
	\begin{align*}
		p_k^\pm\in S^4,\qquad p_k^++p_k^-=\frac{\alpha_0}{\alpha(\lambda_k)}x_0.
	\end{align*}
	Define
	\begin{align}
		b_k=b'_{\frac12}\oplus b_{\lambda_k}^{p_k^+}\oplus b_{\lambda_k}^{p_k^-},\qquad b_\infty=b'_{\frac12}\oplus0\oplus0.
	\end{align}
	Then $b_k$ and $b_\infty$ are Yang-Mills connections on $E'$. Since the squared curvature norm is additive under direct sums, \eqref{YM(b'_lambda;S^4)} and \eqref{int_S^4theta|F_b_lambda|^2dTheta and YM(b'_lambda;S^4)} give
	\begin{align*}
		\int_{S^4}x|F_{b_k}|^2d\Theta&=C_0\left(-\alpha_0x_0+\alpha(\lambda_k)(p_k^++p_k^-)\right)=0,\\
		\int_{S^4}x|F_{b_\infty}|^2d\Theta&=-C_0\alpha_0x_0\ne0.
	\end{align*}
	By Proposition \ref{lem:stationary}, each $\pi^*b_k$ is a stationary weak Yang-Mills connection over $B^5$, whereas $\pi^*b_\infty$ is not stationary.
	
	Now we prove the convergence. Since the $L^2$ norms of a form and its covariant derivative are invariant under rotations, we have
	\begin{align*}
		&\|b_{\lambda_k}^{p_k^\pm}\|_{L^2(S^4)}=\|b'_{\lambda_k}\|_{L^2(S^4)}\le C\lambda_k\to0\textrm{ as }k\to\infty,\\
		&\sup_k\|\nabla b'_{\lambda_k}\|_{L^2}=\sup_k\|\nabla b_{\lambda_k}^{p_k^\pm}\|_{L^2}<+\infty,
	\end{align*} 
	Using the same argument as in Lemma \ref{lem:b'_lambda convergence in L^2 and W^1,2}, we have
	\begin{align*}
		b_k\to b_\infty\ \text{in }L^2(S^4),\quad b_k\rightharpoonup b_\infty\ in\ W^{1,2}(S^4).
	\end{align*}
	The $L^1$ estimate in the proof of Lemma \ref{lem:F_b to0 in L^2} and rotation invariance give
	\begin{align*}
		\|F_{b_{\lambda_k}^{p_k^\pm}}\|_{L^1(S^4)}=\|F_{b'_{\lambda_k}}\|_{L^1(S^4)}\le C\lambda_k^2(1+|\ln\lambda_k|)\to0.
	\end{align*}
	Together with the uniform $L^2$ bound on the curvatures, this implies $F_{b_{\lambda_k}^{p_k^\pm}}\rightharpoonup0$ in $L^2(S^4)$. Hence
	\begin{align}
		\pi^*b_k\to\pi^*b_\infty\ in\ L^2(B^5),\quad\pi^*b_k\rightharpoonup\pi^*b_\infty\ in\ W^{1,2}(B^5),\quad F_{\pi^*b_k}\rightharpoonup F_{\pi^*b_\infty}\ \text{in }L^2(B^5).
	\end{align}
	The $L^4$ estimate in \eqref{|b_lambda|_L^2}, rotation invariance, and \eqref{pi^*b L^p} also give $\sup_k\|\pi^*b_k\|_{L^4(B^5)}<\infty$. All these connections are smooth away from the origin. This completes the proof.
	\end{proof}
	
	\appendix
	\section*{Use of AI}
	The construction in this paper was inspired by the work of Ding, Li and Li \cite{DLL}. The connections $a$ and $\hat a$, as well as the energy-concentrating family $b'_\lambda$ and its rotated versions $b_\lambda^p$, are independently constructed by the authors. The subsequent direct-sum construction used in the proof of Theorem \ref{thm:main} to cancel the first moment of the curvature energy density was generated by AI. The authors take full responsibility for the mathematical content and conclusions of this paper.
	\section*{Acknowledge}
	The first author is supported by National Key R\&D Program of China 2022YFA1005400. The
	second author is supported by NSFC No. 12531002, 12431004, 11721101. 
	
\end{document}